\documentclass[a4paper,twopage,reqno,12pt]{amsart}
\usepackage[top=30mm,right=30mm,bottom=30mm,left=30mm]{geometry}

\usepackage{pdfsync}
\usepackage{graphicx}
\usepackage{amsmath}
\usepackage{amssymb}
\usepackage{amsfonts}
\usepackage{amsthm}
\usepackage{amstext}
\usepackage{amsbsy}
\usepackage{amsopn}
\usepackage{amscd}
\usepackage{enumerate}
\usepackage{color}
\usepackage[colorlinks]{hyperref}
\usepackage[hyperpageref]{backref}
\usepackage{multirow}
\usepackage{lipsum}
\usepackage{rotating}
\usepackage{longtable}
\usepackage{float}
\usepackage{lscape}
\usepackage{pdflscape}
\usepackage[]{algorithm2e}

\newtheorem{theorem}{Theorem}[section]
\newtheorem{corollary}[theorem]{Corollary}

\newtheorem{lemma}[theorem]{Lemma}
\newtheorem{proposition}[theorem]{Proposition}

\theoremstyle{definition}

\numberwithin{equation}{section}

\newcommand{\GU}{\mathrm{GU}}

\newcommand{\GL}{\mathrm{GL}}
\newcommand{\SL}{\mathrm{SL}}

\newcommand{\Sp}{\mathrm{Sp}}

\newcommand{\PSO}{\mathrm{PSO}}

\newcommand{\SU}{\mathrm{SU}}

\renewcommand{\O}{\mathrm{O}}

\newcommand{\PSL}{\mathrm{PSL}}
\newcommand{\PSU}{\mathrm{PSU}}
\newcommand{\U}{\mathrm{PSU}}
\newcommand{\PGU}{\mathrm{PGU}}
\newcommand{\PSp}{\mathrm{PSp}}

\newcommand{\PGL}{\mathrm{PGL}}
\newcommand{\PGaL}{\mathrm{P\Gamma L}}

\newcommand{\POm}{\mathrm{P \Omega}}

\newcommand{\Alt}{\mathrm{Alt}}
\newcommand{\W}{\mathrm{W}}

\newcommand{\A}{\mathrm{A}}
\newcommand{\Sym}{\mathrm{Sym}}
\renewcommand{\S}{\mathrm{Sym}}

\newcommand{\C}{\mathrm{C}}
\newcommand{\D}{\mathrm{D}}

\newcommand{\Aut}{\mathrm{Aut}}

\newcommand{\Out}{\mathrm{Out}}

\newcommand{\PG}{\mathrm{PG}}

\newcommand{\N}{\mathrm{N}}
\newcommand{\M}{\mathrm{M}}
\newcommand{\F}{\mathbb{F}}

\renewcommand{\W}{\mathrm{W}}
\newcommand{\Dmc}{\mathcal{D}}
\newcommand{\Bmc}{\mathcal{B}}

\newcommand{\Pmc}{\mathcal{P}}
\newcommand{\Cmc}{\mathcal{C}}
\newcommand{\Smc}{\mathcal{S}}

\newcommand{\e}{\epsilon}

\renewcommand{\leq}{\leqslant}
\renewcommand{\geq}{\geqslant}

\renewcommand{\mod}[1]{\ (\mathrm{mod}{\ #1})}
\newcommand{\imod}[1]{\allowbreak\mkern4mu({\operator@font mod}\,\,#1)}

\begin{document}
 \title[]{Flag-transitive block designs and unitary groups}

 \author[S.H. Alavi]{Seyed Hassan Alavi}%
 \thanks{Corresponding author: S.H. Alavi}
 \address{Seyed Hassan Alavi, Department of Mathematics, Faculty of Science, Bu-Ali Sina University, Hamedan, Iran.
 }%
 \email{alavi.s.hassan@basu.ac.ir and  alavi.s.hassan@gmail.com (G-mail is preferred)}
 \author[M. Bayat]{Mohsen Bayat}%
 \address{Mohsen Bayat, Department of Mathematics, Faculty of Science, Bu-Ali Sina University, Hamedan, Iran.}%
 \email{m.bayat@basu.ac.ir and mohsen0sayeq24@gmail.com}
 \author[A. Daneshkhah]{Asharf Daneshkhah}%
 \address{Asharf Daneshkhah, Department of Mathematics, Faculty of Science, Bu-Ali Sina University, Hamedan, Iran.
 }%
  \email{adanesh@basu.ac.ir}

 \date{\today}%

\begin{abstract}
In this article, we study $2$-designs with $\gcd(r, \lambda)=1$ admitting a flag-transitive automorphism group. The automorphism groups of these designs are point-primitive  of almost simple or affine type. We determine all pairs $(\Dmc, G)$, where $\Dmc$ is a $2$-design with $\gcd(r, \lambda)=1$ and $G$ is a flag-transitive almost simple automorphism group of $\Dmc$ whose socle is $X=\PSU(n, q)$ with $(n, q)\neq (3, 2)$ and prove that such a design belongs to one of the two infinite families of Hermitian unitals and Witt-Bose-Shrikhande spaces, or it is isomorphic to a design with parameters $(6, 3, 2)$, $(7, 3, 1)$, $(8, 4, 3)$, $(10, 6, 5)$, $(11, 5, 2)$ or $(28, 7, 2)$.
\end{abstract}

\maketitle

\section{Introduction}\label{sec:intro}

A $2$-design $\Dmc$ with parameters $(v, b, r, k, \lambda)$ is a pair $(\Pmc, \Bmc)$ with a set $\Pmc$ of $v$ points and a set $\Bmc$ of $b$ blocks such that each block is a $k$-subset of $\Pmc$ and each two distinct points are contained in $\lambda$ blocks. We say $\Dmc$ is nontrivial if $2 < k < v-1$, and \emph{symmetric} if $v = b$. Each point of $\Dmc$ is contained in exactly $r$ blocks which is called the \emph{replication number} of $\Dmc$. An \emph{automorphism} of a $2$-design $\Dmc$ is a permutation of the points permuting the blocks and preserving the incidence relation. The full automorphism group $\Aut(\Dmc)$ of $\Dmc$ is the group consisting of all automorphisms of $\Dmc$. A \emph{flag} of $\Dmc$ is a point-block pair $(\alpha, B)$ such that $\alpha \in B$. For $G\leq \Aut(\Dmc)$, $G$ is called \emph{flag-transitive} if $G$ acts transitively on the set of flags. The group $G$ is said to be \emph{point-primitive} if $G$ acts primitively on $\Pmc$. A group $G$ is said to be \emph{almost simple} with socle $X$ if $X\unlhd G\leq \Aut(X)$, where $X$ is a nonabelian simple group. We here adopt the standard notation for finite simple groups of Lie type, for example, we use $\PSL(n, q)$, $\PSp(n, q)$, $\PSU(n, q)$ and $\POm^{\e}(n, q)$ with $\e \in \{\circ, -, +\}$ to denote the finite classical simple groups. Symmetric and alternating groups on $n$ letters are denoted by $\Sym_{n}$ and $\Alt_{n}$, respectively. Also for a given positive integer $n$ and a prime divisor $p$ of $n$, we denote the $p$-part of $n$ by $n_{p}$, that is to say, $n_{p}=p^{t}$ with $p^{t}\mid n$ but $p^{t+1}\nmid n$. Further notation and definitions in both design theory and group theory are standard and can be found, for example in~\cite{b:beth1999design,b:Atlas,b:Dixon,b:Hugh-design}.

The main aim of this paper is to study $2$-designs with $\gcd(r, \lambda)=1$ admitting a flag-transitive automorphism group $G$. According to a result of \cite[2.3.7]{b:dembowski}, any flag-transitive group $G$ must acts point-primitively on $\Dmc$. In 1988, Zieschang \cite{a:ZIESCHANG} proved that if an automorphism group $G$ of a $2$-design with $\gcd(r,\lambda)=1$ is flag-transitive, then $G$ is point-primitive group of almost simple or affine type. Such designs admitting an almost simple automorphism group with socle being an alternating group, a sporadic simple group or a finite simple exceptional group have been studied in \cite{a:ABD-Exp, a:A-Exp, a:Zhou-lam-large-sporadic,a:Zhan-nonsym-sprodic,a:Zhou-nonsym-alt,a:Zhu-sym-alternating}. This problem for the case where the socle is a finite simple classical group of Lie type is still open. This paper is devoted to studying $2$-designs with $\gcd(r, \lambda)=1$ admitting a flag-transitive almost simple automorphism group $G$ whose socle is $\PSU(n, q)$ with $(n, q)\neq (3, 2)$. We know two infinite families of examples of designs with $\gcd(r, \lambda)=1$ namely Hermitian unitals with parameters $(q^{3}+1,q+1,1)$ and Witt-Bose-Shrikhande space $\W(2^n)$ with parameters $(2^{n-1}(2^{n}-1), 2^{n-1}, 1)$. The Hermitian unitals are examples of designs with $2$-transitive automorphism groups \cite{a:kantor-Homdes} and the latter example arises from studying flag-transitive linear spaces \cite{a:delan-linear-space}. Our main result is Theorem \ref{thm:main} below.

\begin{theorem}\label{thm:main}
Let $\Dmc$ be a nontrivial $2$-design with $\gcd(r, \lambda)=1$, and let $\alpha$ be a point of $\Dmc$. Suppose that $G$ is an automorphism group of $\Dmc$ whose socle is $X=\PSU(n, q)$ with $(n, q)\neq (3, 2)$. If $G$ is flag-transitive,  then $\lambda \in$ $\{1, 2, 3, 5\}$ and $v$, $k$, $\lambda$, $X$, $G_{\alpha}\cap X$ and $G$ are as in one of the lines in {\rm Table~\ref{tbl:main}} or one of the following holds:
\begin{enumerate}[\rm (a)]
\item $\Dmc$ is a Hermitian unital with parameters $(q^{3}+1,q+1,1)$ and $X$ is $\PSU(3, q)$.
\item $\Dmc$ is the Witt-Bose-Shrikhande space with parameters $(2^{n-1}(2^{n}-1), 2^{n-1},1)$ for $n\geq 3$ and $X$ is $\PSU(2, 2^n)$;
\end{enumerate}
\end{theorem}

\begin{table}
  \centering\small
  \caption{Some nontrivial $2$-design with $\gcd(r,\lambda)=1$.}\label{tbl:main}
  \begin{tabular}{lllllllllll}
    \hline
    Line &
    $v$ &
    $b$ &
    $r$ &
    $k$ &
    $\lambda$ &
    $X$ &
    $G_{\alpha}\cap X$&
    $G$ &
    Designs &
    References \\ \hline
    $1$ &
    $6$ &
    $10$ &
    $5$ &
    $3$ &
    $2$ &
    $\PSU(2, 5)$ &
    $\D_{10}$ &
    $\PSU(2, 5)$ &
    - &
    \cite{b:Handbook, a:Non-symmetric}\\
    $2$ &
    $7$ &
    $7$ &
    $3$ &
    $3$ &
    $1$ &
    $\PSU(2, 7)$ &
    $\Sym_{4}$ &
    $\PSU(2, 7)$ &
    $\PG(2, 2)$ &
    \cite{a:ABD-PSL2, b:Handbook, a:reg-reduction} \\
    $3$ &
    $8$ &
    $14$ &
    $7$ &
    $4$ &
    $3$ &
    $\PSU(2, 7)$&
    $7{:}3$ &
    $\PSU(2, 7)$ &
    - &
    \\
    $4$ &
    $10$ &
    $15$ &
    $9$ &
    $6$ &
    $5$ &
    $\PSU(2, 9)$&
    $3^2:4$&
    $\PSU(2, 9)$ &
    -&
     \cite{a:ABD-PSL2, b:Handbook, a:Non-symmetric} \\
    $5$ &
    $11$ &
    $11$ &
    $5$ &
    $5$ &
    $2$ &
    $\PSU(2, 11)$&
    $\Alt_{5}$&
    $\PSU(2, 11)$ &
    Hadamard &
    \cite{a:ABD-PSL2, b:Handbook, a:reg-reduction} \\
    $6$ &
    $28$ &
    $36$ &
    $9$ &
    $7$ &
    $2$ &
    $\PSU(2, 8)$&
    $\D_{18}$&
    $\PSU(2, 8)$ &
    - &
    \cite{b:Handbook, a:Non-symmetric} \\
    \hline
    \scriptsize
    Note: & \multicolumn{10}{p{12cm}}{\tiny The last column addresses to references in which a design with the parameters in the line has been constructed.}
  \end{tabular}
\end{table}

It is worth noting that to our knowledge for symmetric designs, we only know two designs with $\gcd(k,\lambda)=1$, namely those in lines 2 and 5 of Table~\ref{tbl:main} \cite{a:ABD-Un,a:ABDZ-U4, a:D-PSU3}. In order to prove Theorem~\ref{thm:main} in Section~\ref{sec:proof}, we observe that  the group $G$ is point-primitive \cite[2.3.7]{b:dembowski}, or equivalently, the point stabiliser $H=G_\alpha$ is maximal in $G$. If $X$ is a finite unitary simple group, then we apply the Aschbacher's Theorem~\cite{a:Aschbacher} which says that $H$ lies in one of the seven geometric families $\Cmc_{i}$ for $i=1,\ldots,7$ of subgroups of $G$, or in the family $\Smc$ of almost simple subgroups with some irreducibility conditions. The case where $X=\PSU(2, q)$ has been separately studied in Proposition~\ref{prop:psu2}. For the case where $n\geq 3$, in Proposition \ref{prop:large}, we obtain possible subgroups $H$ satisfying $|G|\leq |H|\cdot |H|_{p'}^{2}$. We then analyse each of these possible cases and prove the theorem. In this paper, we use the software \textsf{GAP} \cite{GAP4} for computational arguments.

\section{Examples and comments}\label{sec:example}

In this section, we provide some examples of $2$-designs admitting flag-transitive and point-primitive automorphism groups. We, in particular, make some  comments on Theorem~\ref{thm:main} and the designs mentioned in Table~\ref{tbl:main}. We remark here that the designs in Table~\ref{tbl:main} can be found in \cite{a:ABD-PSL2,a:Non-symmetric}, but the construction given here is obtained by \textsf{ GAP} \cite{GAP4}.

The Hermitian unital with parameters $(q^{3}+1,q+1,1)$ is a well-known example of flag-transitive $2$-designs \cite{a:kantor-Homdes}. Let $V$ be a three-dimensional vector space over the field $\F_{q^2}$ with a non-degenerate Hermitian form. The Hermitian unital is an incidence structure whose points are $q^3+1$ totally isotropic $1$-spaces in $V$, the lines are the sets of $q-1$ points lying in a non-degenerate $2$-space, and the incidence is given by inclusion. Any group $G$ with $\PSU(3, q)\leq G\leq \PGU(3, q)$ acts flag-transitively on Hermitian unital design.

The Witt-Bose-Shrikhande space with parameters $(2^{n-1}(2^{n}-1), 2^{n-1},1)$ in part(b) of Theorem \ref{thm:main} can be defined from the group $\PSU(2, q)\cong \PSL(2,q)$ with $q=2^n\geq 8$ \cite{a:delan-linear-space}. In this incidence structure, the points are the dihedral subgroups of order $2q+1$, the lines are the involutions of $\PSU(2, q)$, a point being incident with a line precisely when the dihedral subgroup contains the involution. An almost simple group $G$ with socle $X=\PSU(2,q)$ acts flag-transitively on Witt-Bose-Shrikhande space.

The design $\Dmc=(\Pmc, \Bmc)$ with parameters $(6, 3, 2)$ in line 1 of Table~\ref{tbl:main} is the unique design where $\Pmc=\{1,2,3,4,5,6\}$ and base block $\{1,2,3\}$. The full automorphism group of $\Dmc$ is $\PSU(2,5)\cong \A_{5}$. Note that $\PSU(2,5)$ acts flag-transitively on $\Dmc$ with point-stabiliser $\D_{10}$ and block-stabiliser $\Sym_{3}$.

The design $\Dmc=(\Pmc, \Bmc)$ with parameters $(7, 3, 1)$ in line 2 of Table~\ref{tbl:main} is the unique well-known symmetric design, namely, \emph{Fano Plane} admitting flag-transitive and point-primitive automorphism group $\PSL(2,7)\cong \PSU(2,7)$ with point-stabiliser $\Sym_{4}$.

The design in line 3 of Table \ref{tbl:main} is the unique $(8, 4,3)$ design  on $8$ points with base block $\{1, 2, 3, 5 \}$. The full automorphism group of this design is $2^{3}{:}\PSL(2,7)\cong 2^{3}{:}\PSU(2,7)$ acting flag-transitively and anti-flag-transitively. The automorphism group $\PSU(2,7)$ is flag-transitive with point-stabiliser  $7{:3}$ and block-stabiliser $\Alt_{4}$.

The design on line 4 of Table \ref{tbl:main} is the unique design $\Dmc$ with parameters $(10, 6, 5)$. The base block of this design is $\{1, 2, 3, 4, 5, 6 \}$ and the full automorphism group of $\Dmc$ is $\S_{6}$ with point-stabiliser $\Sym_{3}^{2}{:}2$ and block-stabiliser $2 \times \Sym_{4}$. The automorphism group $\PSU(2,9)$ is flag-transitive with point-stabiliser $3^2:4$ and block-stabiliser $\Sym_{4}$.

The design in line 5 of Table \ref{tbl:main} is the unique symmetric $(11,5, 2)$ design known as a \emph{Paley difference set} which is in fact a Hadamard design with base block $\{ 1, 2, 3, 5, 11 \}$, and its full automorphism group is $\PSU(2,11)$ acting  flag-transitively and point-primitively. In this case, the point-stabiliser is isomorphic to $\A_{5}$.

The design in line 6 of Table \ref{tbl:main} is the unique design $\Dmc$ with parameters $(28,7,2)$. The base block of this design is $\{1, 6, 12, 13, 14, 24, 28  \}$ and the full automorphism group of $\Dmc$ is $\PSU(2,8){:}3$ with point-stabiliser $9{:}6$ and block-stabiliser $7{:}6$. The automorphism group $\PSU(2,8)$ is flag-transitive with  point-stabiliser $\D_{18}$ and block-stabiliser $\D_{14}$.

\section{Preliminaries}\label{sec:pre}

In this section, we state some useful facts in both design theory and group theory. Lemma \ref{lem:New} below is an elementary result on subgroups of almost simple groups.

\begin{lemma}\label{lem:New}{\rm \cite[Lemma 2.2]{a:ABD-PSL3}}
Let $G$  be an almost simple group with socle $X$, and let $H$ be maximal in $G$ not containing $X$. Then $G=HX$ and $|H|$ divides $|\Out(X)|{\cdot}|H\cap X|$.
\end{lemma}

If a group $G$ acts on a set $\Pmc$ and $\alpha\in \Pmc$, the \emph{subdegrees} of $G$ are the size of orbits of the action of the point-stabiliser $G_\alpha$ on $\Pmc$.

\begin{lemma}\label{lem:six} {\rm \cite[Lemmas 5 and 6]{a:Zhou-nonsym-alt}}
Let $\Dmc$ be a $2$-design with prime replication number $r$, and let $G$ be a flag-transitive automorphism group of $\Dmc$. If $\alpha$ is a point in $\Pmc$ and $H:=G_{\alpha}$, then
\begin{enumerate}[\rm \quad (a)]
\item $r(k-1)=\lambda(v-1)$.  In particular, if $\gcd(r, \lambda)=1$, then $r$ divides $v-1$;
\item $vr=bk$;
\item $r\mid |H|$ and $\lambda v<r^2$;
\item $r\mid d$, for all nontrivial subdegrees $d$ of $G$.
\end{enumerate}
\end{lemma}

\begin{corollary}\label{cor:large}
Let $\Dmc$ be a flag-transitive $2$-design with automorphism group $G$. Then $|G|\leq |G_{\alpha}|^3$,  where $\alpha$ is a point in $\Dmc$.
\end{corollary}
\begin{proof}
By Lemma \ref{lem:six}(c), we have that $v<r^2$. The result follows from th  $v{=}|G{:}G_{\alpha}|$ and $r\leq |G_{\alpha}|$.
\end{proof}

\begin{lemma}\label{lem:subdeg}{\rm \cite[3.9]{a:LSS1987}}
If $X$ is a group of Lie type in characteristic $p$, acting on the set of cosets of a maximal parabolic subgroup, and $X$ is neither $\PSL(n, q)$, $\POm^{+}(n, q)$ (with $n/2$ odd), nor $E_{6}(q)$, then there is a unique subdegree which is a power of $p$.
\end{lemma}

In Lemma~\ref{lem:subdegree} below, we collect some information on subdegrees of primitive actions of almost simple group $G$ with socle $X=\PSU(n, q)$ which had also been used in~\cite{a:reg-classical, a:Saxl2002}. Below, we denote by $\,^{\hat{}}H$ the pre-image of the subgroup $H$ of $G$ in the simple group $X$.

\begin{lemma}\label{lem:subdegree}
    Let $G$ be an almost simple group with socle $X=\PSU(n, q)$ for $(n, q)\neq (3, 2)$, and let $H$ be a maximal subgroup of $G$ with $H\cap X$ being as in the second column of {\rm Table~\ref{tbl:subdegree}}. Then the action of $G$ on the cosets of $H$ has subdegrees dividing $d$ listed in the last column of {\rm Table~\ref{tbl:subdegree}}.
\end{lemma}
\begin{proof}
It follows from \cite{a:Saxl2002, a:reg-classical}, \cite[Theorem 2]{a:3-design} and \cite[Table 2]{a:farad}.
\end{proof}

\begin{table}
    \small
    \centering
    \caption{Some subdegrees of $\PSU(n, q)$ acting on the set of cosets of H.}\label{tbl:subdegree}
    \begin{tabular}{cp{5cm}ll}
        \noalign{\smallskip}\hline\noalign{\smallskip}
        $\Cmc_{i}$ &
        \multicolumn{1}{l}{$H\cap X$} &
        \multicolumn{1}{l}{$d$} &
         \multicolumn{1}{l}{Condition}\\
        \hline\noalign{\smallskip}
        $\Cmc_{1}$ & $\,^{\hat{}}\SU(m, q) {\times} \SU(n-m, q){\cdot}(q + 1)$  &  $(q^{m}-(-1)^m)(q^{n-m}-(-1)^{n-m})$ & $n\geq 3$\\
        $\Cmc_{2}$ & $^{\hat{}}\SU(m, q)^t {\cdot} (q+1)^{t-1}{\cdot} \Sym_{t}$ &  $t(t-1)(q^m-(-1)^m)^2$ & $n\geq 3$\\
        $\Cmc_{2}$ & $^{\hat{}}(q+1)^{n-1}{\cdot}\Sym_{n}$ & $n(n-1)(n-2)(q+1)^3{\cdot}2^{-1}$ & $n\geq 3$\\
        $\Cmc_{2}$ & $^{\hat{}}\SL(m, q^2){\cdot}(q-1){\cdot}2$ & $2(q^{n}-1)$ & $n=2m$\\
        $\Cmc_{2}$ & $\D_{2(q-1)/\gcd(2, q-1)}$ &  $q-1$ & $n=2$\\
        $\Cmc_{3}$ & $\D_{2(q+1)/\gcd(2, q-1)}$ &  $q+1$ & $n=2$\\
        \hline\noalign{\smallskip}
    \end{tabular}
\end{table}

\begin{lemma}\label{lem:Tits}{\rm (Tits' Lemma~\cite[1.6]{a:tits})}
If $X$ is a simple group of Lie type in characteristic $p$, then any proper subgroup of index prime to $p$ is contained in a parabolic subgroup of $X$.
\end{lemma}

For a given positive integer $n$ and a prime divisor $p$ of $n$, we denote the $p$-part of $n$ by $n_{p}$, that is to say, $n_{p}=p^{t}$ with $p^{t}\mid n$ but $p^{t+1}\nmid n$.

\begin{lemma}\label{lem:p-part}
Suppose that $\Dmc$ is a $2$-design with $\gcd(r, \lambda)=1$. Let $G$ be a point-primitive, flag-transitive almost simple automorphism group of $\Dmc$ with simple socle $X$ of Lie type in characteristic $p$. If the point-stabiliser $H=G_{\alpha}$ is not a parabolic subgroup of $G$, then $|G|<|H|{\cdot}|H|_{p'}^2$. Moreover, $|X|<|\Out(X)|_{p'}^2{\cdot}|H\cap X|\cdot|H\cap X|_{p'}^2$.
\end{lemma}
\begin{proof}
Since $r$ a divisor of $\lambda(v-1)$ and $\gcd(r, \lambda)=1$, the parameter $r$ must divide $v-1$. Note by Lemma \ref{lem:Tits} that $p$ divides $v= |G:G_{\alpha}|$. Then since also $r$ divides, we conclude that $r$ is a divisor of $|H|_{p'}$. Therefore, $v<r^2$ implies $|G:H|< |H|_{p'}^2$, or equivalently, $|G|<|H|{\cdot}|H|_{p'}^2$. We now apply Lemma \ref{lem:New}, and conclude that $|X|<|\Out(X)|_{p'}^2{\cdot}|H\cap X|\cdot|H\cap X|_{p'}^2$. 
\end{proof}


\begin{lemma}\label{lem:flag}{\rm \cite[2.2.5]{b:dembowski}}
Let $\Dmc$ be a $2$-design. If $\Dmc$ satisfies $r=k+\lambda$ and $\lambda \leq 2$, then $\Dmc$ is embedded in a symmetric $2$-$(v+k+\lambda, k+\lambda, \lambda)$ design.
\end{lemma}

\begin{lemma}\label{lem:symmetric}{\rm \cite[2.3.8]{b:dembowski}}
Let $\Dmc$ be a $2$-design and $G\leq \Aut(\Dmc)$. If $G$ is $2$-transitive on points and $\gcd(r, \lambda)= 1$, then $X$ is flag-transitive.
\end{lemma}

\begin{lemma}\label{lem:divisible}
Suppose that $\Dmc$ is a $2$-design with $\gcd(r, \lambda)=1$. Let $G$ be a  flag-transitive automorphism group of $\Dmc$ with simple socle $X$ of Lie type in characteristic $p$. If the point-stabiliser $H=G_{\alpha}$ contains a normal quasi-simple subgroup $K$ of Lie type in characteristic $p$ and $p$ does not divide $|Z(K)|$, then either $p$ divides $r$, or $K_{B}$ is contained in a parabolic subgroup $P$ of $K$ and $r$ is divisible by $|K{:}P|$.
\end{lemma}
\begin{proof}
If $B$ is a block incident with a point $\alpha$ of $\Dmc$, then $r= |H{:}H_{B}|$, and so $|K{:}K_{B}|$ divides $r$. Note that $\gcd(r, \lambda)=1$. If $\gcd(r, p)=1$, then $|K{:}K_{B}|$ is coprime to $p$, and now Lemma~\ref{lem:Tits} implies that $K_{B}$ is contained in a (maximal) parabolic subgroup $P$ of $K$. Hence $r$ is divisible by $|K{:}P|$.
\end{proof}

\begin{lemma}\cite[Lemma~4.2 and Corollary~4.3]{a:AB-Large-15}\label{lem:up-lo-b}
If $n\geq 3$, then
\begin{align*}
(1-q^{-1})q^{n^2-2}<|\PSU(n, q)|\leq |\GU(n, q)|\leq (1+q^{-1})(1-q^{-2})(1+q^{-3})q^{n^2}.
\end{align*}
\end{lemma}

\begin{lemma}\label{lem:equation}
Let $q$ be a prime power and $m\geq 3$ be a positive integer number, then
\begin{align*}
(q^m-(-1)^{m})(q^{m-1}-(-1)^{m-1}){\cdots}(q^2-1)<q^{\frac{m^2+m-2}{2}}.
\end{align*}
\end{lemma}
\begin{proof}
Suppose first that $m$ is odd. Note that $(q^m+1)(q^{m-1}-1)<q^{2m-1}$. Therefore, $(q^m+1)(q^{m-1}-1){\cdots}(q^2-1)<q^{(2m-1)+(2m-5)+\ldots+5}$. Since $(2m-1)+(2m-5)+\ldots+5$ is an arithmetic sequence with $(m-1)/2$ terms, initial term $5$ and common difference $4$, it follows that  $(q^m-1)(q^{m-1}+1){\cdots}(q^2-1)<q^{(m^2+m-2)/2}$. Suppose now that $m$ is even. As $m-1$ is odd, we have that $(q^m-1)(q^{m-1}+1){\cdots}(q^2-1)<(q^m-1){\cdot}q^{(m^2-m-2)/2}<q^{(m^2+m-2)/2}$.
\end{proof}

For a finite group $X$, let $p(X)$ be the minimal degree of permutation representation of $X$. In particular, for a finite simple group $X$, the integer $p(X)$ is just the index of the largest proper subgroup of $X$, and we know these degrees for all finite simple unitary groups. Here, we need $p(X)$ for finite  simple groups $X=\PSU(n, q)$.

\begin{lemma}\label{lem:min-deg}{\rm \cite{a:ABCD-r-prime}}
The minimal degrees $p(X)$ of permutation representations of $X=\PSU(n, q)$ are given in {\rm Table~\ref{tbl:min-deg}}.
\end{lemma}

\begin{table}
\centering
\small
\caption{The minimal degrees of permutation representations of $X=\PSU(n, q)$.}\label{tbl:min-deg}
\begin{tabular}{lll}
\hline
$X$ & $p(X)$ & Conditions\\\hline
$\PSU(n, q)$ & $(q^{n}-(-1)^{n})(q^{n-1}-(-1)^{n-1})/(q^2-1)$ & $n\geq 5$ and $(n, q)\neq (6s, 2)$\\
$\PSU(n, q)$ & $2^{n-1}(2^n-1)/3$ & $n\equiv 0 \mod{6}$ \\
$\PSU(4, q)$ & $(q+1)(q^3+1)$ & \\
$\PSU(3, q)$ & $q^3+1$ & $q\neq 5$\\
$\PSU(3, 5)$ & $50$ & \\
\hline
\end{tabular}
\end{table}

The maximal subgroups $H$ of almost simple groups $G$ with socle $\PSU(n,q)$ have been determined by Aschbacher~\cite{a:Aschbacher}, and  so the subgroup $H$ lies in one of the seven geometric families $\Cmc_{i}$ for $i=1, 2, 3, 4, 5, 6, 7$ of subgroups of $G$, or in the family $\Smc$ of almost simple subgroups with some irreducibility conditions. We follow the description of these subgroups as in \cite{b:KL-90}. In what follows, if $H$ belongs to the family $\Cmc_{i}$, for some $i$, then we sometimes say that $H$ is a $\Cmc_{i}$-subgroup. A rough description of the $\Cmc_i$ families is given in Table \ref{t:max}. We also denote by $\,^{\hat{}}H$ the pre-image of the group $H$ in the corresponding linear group.

\begin{lemma}\label{lem:p-part-out}
Let $G$ be an almost simple group with socle $X=\PSU(n,q)$, where $n\geq 3$ and $q=p^{a}$, and let $H$ be a maximal geometric subgroup of $G$ with $H \notin \Cmc_6$. If $|H\cap X|_{p}<|\Out(X)|$, then one of the following holds
\begin{enumerate}[\rm (a)]
\item $H$ is a $\Cmc_2$-subgroup of type $\GU(1, q)\wr\S_n$;
\item $H$ is a $\Cmc_2$-subgroup of type $\GL(2,9)$;
\item $H$ is a $\Cmc_3$-subgroup of type $\GU(1, q^n)$.
\end{enumerate}
\end{lemma}
\begin{proof}
Suppose that $X=\PSU(n, q)$ with $n\geq 3$ and $q=p^{a}$. Since $H$ is geometric maximal subgrup in $G$, then by Aschbacher's Theorem~\cite{a:Aschbacher}, the subgroup $H$ lies in one of the families $\Cmc_{i}$ for some $i=1, 2, 3, \ldots, 7$. Let $H \notin \Cmc_6$. We will analyse each of these cases separately. \smallskip

\noindent(1) If $H\in \Cmc_{1}$, then $H$ is reducible, and so it is either a parabolic subgroup $P_{m}$, or the stabilizer $N_m$ of a non-singular subspace. Then by \cite[Propositions 4.1.4 and 4.1.18 ]{b:KL-90}, $|H\cap X|_{p}=q^{n(n-1)/2}$. Since $|\Out(X)|=2a\cdot \gcd(n, q+1)$, the inequality $|H\cap X|_{p}<|\Out(X)|$ implies that  $q^{n(n-1)/2}<2a\cdot \gcd(n, q+1)$. Note that $n\geq 3$ and $2q>q+1$. Thus $q^{n-1}<4a$, which is impossible.\smallskip

\noindent(2) Let $H \in \Cmc_2$. Then $H$ preserves a partition $V =V_1\oplus \cdots \oplus V_t$ with each $V_i$'s of dimension $m$, so $mt=n$ and either the $V_i$'s are non-singular and the
partition is orthogonal, or $t=2$ and the $V_i$'s are totally singular.\smallskip

Assume first that the $V_i$'s are non-singular. Here by \cite[Proposition 4.2.9]{b:KL-90}, $H\cap X$ is isomorphic to $\,^{\hat{}}\SU(m, q)^t \cdot (q+1)^{t-1}\cdot \Sym_{t}$. Therefore, $|H\cap X|_{p}\geq q^{mt(m-1)/2}$. Then the inequality $|H\cap X|_{p}<|\Out(X)|$ implies that $q^{mt(m-1)/2}<2a\cdot \gcd(n, q+1)$. If $m\geq 3$, then $q^{n-1}<4a$, which is impossible. If $m=2$, then $q^t<2a\cdot \gcd(2t, q+1)$. Thus $q^t<2at\cdot \gcd(2, q+1)$. This inequality does not holds for any $q=p^a$ and $t$. Therefore, $m=1$, and hence $H$ is a $\Cmc_2$-subgroup of type $\GU(1, q)\wr\S_n$, and this part (a).\smallskip

Assume now that $t=2$ and both the $V_i$'s are singular. In this case by \cite[Proposition 4.2.4]{b:KL-90}, $H\cap X$ is isomorphic to $\,^{\hat{}}\SL(m, q^2)\cdot(q-1)\cdot 2$, where $2m=n$, and so $|H\cap X|_{p}\geq q^{m(m-1)}$.  Then by the inequality $|H\cap X|_{p}<|\Out(X)|$, we have that $q^{m(m-1)}<2a\cdot \gcd(n, q+1)$. Recall that $2m=n$ and $m\geq 2$. If $m\geq 3$, then $q^5<4a$, which is impossible. If $m=2$, then $q<2a\cdot \gcd(4, q+1)$. This inequality holds only for $q=3$. This follows part (b).\smallskip

\noindent(3)  Let $H \in \Cmc_3$. Then $H$ is a field extension group for some field extension of $\F_q$ of odd degree $t$. In this case by \cite[Proposition 4.3.6]{b:KL-90}, $H\cap X$ is isomorphic to $\,^{\hat{}}\SU(m, q^t)\cdot(q^t+1)\cdot t/(q+1)$. Then the inequality $|H\cap X|_{p}<|\Out(X)|$ yields that $q^{mt(m-1)/2}<2a\cdot \gcd(n, q+1)$. Since $2q>q+1$ and $t\geq 3$, it follows that $q^{3m(m-1)/2}<4aq$. If $m\geq 2$, then $q^{2}<4a$, which is impossible.Therefore, $m=1$, and this is part (c).\smallskip

\noindent(4) If $H \in \Cmc_4$, then $H$ is the stabilizer of the tensor product of two non-singular spaces of dimensions $m>t>1$. Here \cite[Proposition 4.4.10]{b:KL-90} implies that $|H\cap X|_{p}\geq q^{(m^2-m+t^2-t)/2}$. Therefore, $q^{(m^2-m+t^2-t)/2}<2a\cdot \gcd(n, q+1)$ by the inequality $|H\cap X|_{p}<|\Out(X)|$. Since $t^2-t\geq 2$ and $2q>q+1$, we have that $q^{m(m-1)/2}<4a$. Thus $q^3<4a$, which is impossible.\smallskip

\noindent(5) Let $H \in \Cmc_5$. In this case $H$ is a subfield group. Then $H$ is a unitary group of dimension $n$ over $\F_{q_0}$, where $q=q_{0}^t$ with $t$ an odd prime or one of the following holds.
\begin{enumerate}[\rm(i)]
\item If $H$ is a unitary group of dimension $n$ over $\F_{q_0}$, where $q=q_{0}^t$ with $t$ an odd prime. Then by \cite[propositions 4.5.3]{b:KL-90}, $|H\cap X|_{p}\geq q_{0}^{n(n-1)/2}$. Since $|\Out(X)|=2a\gcd(n, q+1)$, we must have $q_{0}^{n(n-1)/2}<2a\gcd(n, q+1)$. Note that $n\geq 3$. Thus $q_{0}^{n}<2an$, which is impossible.

\item If $H\cap X$ is isomorphic to $\PSO(n, q)^{\epsilon}.2$ with $n$ even and $q$ odd, then $|H\cap X|_{p}\geq q^{m^2-m}$, where $2m=n$ and $m\geq 2$. Therefore, $q^{m^2-m}<4a\cdot \gcd(m, q+1)$ by the inequality $|H\cap X|_{p}<|\Out(X)|$, and so $q^{m^2-m}<4am$, which is impossible.

\item Finally let $H$ is isomporphic to $N(\PSp(n, q))$, with $n$ even. In this case by~\cite[Proposition 4.5.6]{b:KL-90}, $H\cap X$ is isomorphic to $\,^{\hat{}}\Sp(n, q){\cdot} \gcd(n/2, q+1)$, and so the inequality $|H\cap X|_{p}<|\Out(X)|$ implies that $q^{m^2}<2a\cdot \gcd(2m, q+1)$, where $2m=n$. Thus $q^{m^2}<4am$, which is impossible.
\end{enumerate}

\noindent(6) If $H\in \Cmc_7$, then $H$ is a symmetric tensor product group. Here by \cite[proposition 4.7.3]{b:KL-90}, we have that $|H\cap X|_{p}\geq q^{mt(m-1)/2}$, where $m\geq 3$ and $t\geq 3$. Then the inequality $|H\cap X|_{p}<|\Out(X)|$ implies that $q^{mt(m-1)/2}<2a\gcd(n, q+1)$, and so $q^{n-1}<4a$, which is impossible.
\end{proof}

\begin{table}
\small
\caption{The geometric subgroup collections}\label{t:max}
\begin{tabular}{clllll}
\hline
Class & Rough description\\
\hline
$\Cmc_1$ & Stabilisers of subspaces of $V$\\
$\Cmc_2$ & Stabilisers of decompositions $V=\bigoplus_{i=1}^{t}V_i$, where $\dim V_i = m$\\
$\Cmc_3$ & Stabilisers of prime index extension fields of $\F$\\
$\Cmc_4$ & Stabilisers of decompositions $V=V_1 \otimes V_2$\\
$\Cmc_5$ & Stabilisers of prime index subfields of $\F$\\
$\Cmc_6$ & Normalisers of symplectic-type $r$-groups in absolutely irreducible representations\\
$\Cmc_7$ & Stabilisers of decompositions $V=\bigotimes_{i=1}^{t}V_i$, where $\dim V_i = m$\\
\hline
\end{tabular}
\end{table}

\section{Proof of the main result}\label{sec:proof}
In this section, we prove Theorem \ref{thm:main}. Suppose that $\Dmc$ is a nontrivial $2$-design with $\gcd(r, \lambda)=1$ and $G$ is an automorphism group of $\Dmc$ which is almost simple with socle $X$ being a finite non-abelian unitary simple group. Suppose now that $G$ is flag-transitive. Then \cite[2.3.7]{b:dembowski} implies that $G$ is point-primitive. Let $H=G_{\alpha}$, where $\alpha$ is a point of $\Dmc$. Therefore, $H$ is maximal in $G$ (see \cite[7, Corollary 1.5A]{b:Dixon}), and so Lemma~\ref{lem:New} implies that
\begin{align}
   v=\frac{|X|}{|H \cap X|}.\label{eq:v}
\end{align}

We now apply Aschbacher's Theorem~\cite{a:Aschbacher} and conclude that the subgroup $H$ lies in one of the seven geometric families $\Cmc_{i}$ for $i=1, 2, 3, 4, 5, 6, 7$ of subgroups of $G$, or belongs to the family $\Smc$. We analyse each of these cases separately, and we first have a reduction for possible subgroups $H$ in Proposition~\ref{prop:large} below.

\begin{proposition}\label{prop:large}
Suppose that $\Dmc$ is a $2$-design with $\gcd(r, \lambda)=1$. Let $G$ be a point-primitive and flag-transitive almost simple automorphism group of $\Dmc$ with socle $X=\PSU(n,q)$ with $n\geq 3$ and $q=p^{a}$, and let $H$ be a point-stabiliser subgroup of $G$. Then $H$ is a maximal geometric subgroup of one of the following types
\begin{enumerate}[\rm (a)]
\item $H \in \Cmc_{1}$;
\item $H$ is a $\Cmc_2$-subgroup of type $\GU(m, q)\wr\S_t$ with $m=1$ or $m\geq2$ and $t\leq 11$;
\item $H$ is a $\Cmc_2$-subgroup of type $\GL(n/2, q^2)$;
\item $H$ is a $\Cmc_3$-subgroup of type $\GU(m, q^t)$ with $mt=n$;
\item $H$ is a $\Cmc_5$-subgroup of type $\GU(n, q_{0})$ with $q=q_{0}^3$;
\item $H$ is a $\Cmc_5$-subgroup of type $\Sp(n, q)$ or $\O^{\e}(n, q)$ with $\e {\in} \{\circ, -, +\}$.
\end{enumerate}
\end{proposition}
\begin{proof}
Suppose that $X=\PSU(n, q)$ with $n\geq 3$ and $q=p^{a}$. Since $H$ is maximal in $G$, then by Aschbacher's Theorem~\cite{a:Aschbacher}, the subgroup $H$ lies in one of the geometric families $\Cmc_{i}$ of subgroups of $G$, or in the set $\Smc$ of almost simple subgroups not contained in any of these geometric  families. Let $H$ be a non-parabolic subgroup of $G$, that is to say, $H$ is not a $\Cmc_{1}$-subgroup. We now discuss each of these possible cases. We note by Lemma \ref{lem:p-part} that $|G|<|H|{\cdot}|H|_{p'}^2$ in all possible cases.

If $H$ belongs to $\Cmc_{6}$, then by \cite[propositions 4.6.5 and 4.6.6]{b:KL-90} and the inequality $|G|<|H|{\cdot}|H|_{p'}^2$, we only have to consider the pairs $(X, H\cap X)$ listed in Table \ref{tbl:class-s}. For each such $H\cap X$, by~\eqref{eq:v}, we obtain $v$ as in the third column of Table \ref{tbl:class-s}. Moreover, Lemma \ref{lem:six}(a)-(c) says that $r$ divides $\gcd(|H|, v-1)$, and so we can find an upper bound $u_r$ of $r$ as in the fourth column of Table \ref{tbl:class-s}. Then the inequality $\lambda v< r^2$ rules out all these possibilities.

If $H$ is $\Smc$-subgroup, then for $n\leq 12$, the subgroups $H$ are listed in~\cite[Chapter 8]{b:BHR-Max-Low}. Since $|G|<|H|{\cdot}|H|_{p'}^2$, we only have to consider the pairs $(X, H\cap X)$ listed in Table \ref{tbl:class-s}. For each such $H\cap X$, by~\eqref{eq:v}, we obtain $v$ as in the third column of Table \ref{tbl:class-s}, and since $r$ divides $\gcd(|H|, v-1)$ , we can find an upper bound $u_r$ of $r$ as in the fourth column of Table \ref{tbl:class-s}. The inequality $\lambda v< r^2$ rules out all these possible cases except for $\PSL(2, 7)<\PSU(3, 3)$. In this last case, $r=7$, $v=36$, $k=6$ and $b=42$, but in this case by ~\cite[p.14]{b:Atlas} $G=\PSU(3,3)$ does not have any subgroup of index $42$. If $n\geq 14$, by~\cite[Theorem~4.2]{a:Liebeck1985}, we have $|G|<|H|^3$. Hence $n=13$. By~\cite[Theorem~4.2]{a:Liebeck1985}, we have $|H|$ bounded above by $q^{4n+8}$. Since$|G|<|H|{\cdot}|H|_{p'}^2$, we deduce that $|H|_{p'}$ is bounded below by $q^{53}$. It is now easy to rule out all the possible almost simple groups $H$ using the methods of~\cite{a:Liebeck1985}.

Therefore, $H$ is neither a $\Cmc_{6}$-subgroup, nor a $\Smc$-subgroup. If $|\Out(X)|_{p'} \leq |H\cap X|_p$, then since $|X|<|\Out(X)|_{p'}^2{\cdot}|H\cap X|\cdot|H\cap X|_{p'}^2$ by Lemma~\ref{lem:p-part}, we conclude that $|X|<|H\cap X|^3$, and hence the subgroups $H$ satisfying this condition can be read off from~\cite[Theorem 7 and Proposition 4.7]{a:AB-Large-15}. If on the other hand $|\Out(X)|_{p'} > |H\cap X|_p$, then we apply Lemma \ref{lem:p-part-out}, and the result follows.
\end{proof}

\begin{table}
\small
\caption{Some maximal subgroups of $X=\PSU(n,q)$.}\label{tbl:class-s}
\begin{tabular}{cllll}
\hline
Class &
$X$ &
$H\cap X$ &
$v$ &
$u_r$\\
\hline
$\Smc$ &
$\PSU(3,3)$ &
$\PSL(2,7)$ &
$36$ &
$7$
\\
$\Smc$ &
$\PSU(3,5)$&
$\PSL(2,7)$ &
$2{\cdot} 3{\cdot} 5^{3}$ &
$7$
\\
$\Smc$ &
$\PSU(3,5)$&
$\A_{7}$ &
$2{\cdot} 5^{2}$ &
$7$
\\
$\Smc$ &
$\PSU(3,5)$&
$\M_{10}$ &
$5^{2}{\cdot} 7$ &
$5$
\\
$\Cmc_{6}$ &
$\PSU(4,3)$ &
$2^{4}.\A_{6}$ &
$3^{4}{\cdot} 7$ &
$5$
\\
$\Smc$ &
$\PSU(4,3)$ &
$\PSL(3,4)$ &
$2{\cdot} 3^{4}$ &
$7$
\\
$\Smc$ &
$\PSU(4,3)$&
$\A_{7}$ &
$2^{4}{\cdot} 3^{4}$ &
$7$
\\
$\Smc$ &
$\PSU(4,5)$&
$\A_{7}$ &
$2^{4}{\cdot} 3^{2}{\cdot} 5^{5}{\cdot} 13$ &
$7$
\\
$\Smc$ &
$\PSU(4,5)$ &
$\PSU(4,2)$  &
$2{\cdot} 5^{5}{\cdot} 7{\cdot} 13$ &
$5$
\\
$\Cmc_{6}$ &
$\PSU(4,7)$ &
$2^{4}.\Sp_{4}(2)$ &
$2^{2}{\cdot} 5{\cdot} 7^{6}{\cdot} 43$ &
$5$
\\
$\Smc$ &
$\PSU(5,2)$ &
$\PSL(2,11)$  &
$2^{8}{\cdot} 3^{4}$ &
$11$
\\
$\Smc$ &
$\PSU(6,2)$ &
$\M_{22}$&
$2^{8}{\cdot} 3^{4}$ &
$11$
\\
$\Smc$ &
$\PSU(6,2)$&
$\PSU(4,3).2$  &
$2^{9}{\cdot} 3^{3}{\cdot} 5 {\cdot} 11$ &
$7$
\\
\hline
\end{tabular}
\end{table}

\begin{proposition}\label{prop:psu2}
Let $\Dmc$ be a nontrivial $2$-design with $\gcd(r, \lambda)=1$. Suppose that $G$ is an automorphism group of $\Dmc$ of almost simple type with socle $X=\PSU(2, q)$. If $G$ is flag-transitive, then $\Dmc$ is the Witt-Bose-Shrikhande space $\W(2^n)$ with parameters $(2^{n-1}(2^{n}-1), 2^{n-1},1)$ for $n\geq 3$ and $X$ is $\PSU(2, 2^n)$ or  $(v, b, r, k, \lambda)$, $G$, $X$ and $G_{\alpha}\cap X$ are as in lines $1$-$6$ of {\rm Table~\ref{tbl:main}}.
\end{proposition}
\begin{proof}
Suppose that $\Dmc$ is a $2$-design with $\gcd(r, \lambda)=1$ and that $G$ is an almost simple group with socle $X=\PSU(2, q)$. Applying \cite[Theorem 1.1]{a:ABD-PSL2} and \cite[Main Theorem]{a:Saxl2002}, then we can focus on non-symmetric designs with $\lambda\geq 2$. Also, we exclude the case where $q= 4,5$ or $9$ by \cite[Theorem 1.1]{a:Zhou-nonsym-alt}. If $G$ is a flag-transitive automorphism group of $\Dmc$, then by Lemma \cite[2.3.7]{b:dembowski} the point-stabiliser $H=G_{\alpha}$ is maximal in $G$. Let $H_{0}=H\cap X$. Then by~\cite[Theorems 1.1 and 2.2]{a:gmax} either $(G, H, X)$ is as in {\rm Table~\ref{tbl:max-psu}}, or $H_{0}$ is maximal in $X$. In the latter case, $H_{0}$ is (isomorphic to) one of the following groups:

\begin{enumerate}[{\rm \quad(1)}]
\item $\PGU(2, q_{0})$, for $q=q_{0}^2$ odd;
\item $\PSU(2, q_{0})$, for $q=q_{0}^t$ odd, where $t$ is an odd prime number;
\item $\PGU(2, q_{0})$, for $q=2^a=q_{0}^t$, where $t$ is prime and $q_{0}\neq 2$;
\item $\Sym_{4}$, for $q = p\equiv \pm 1 \mod{8}$;
\item $\Alt_{4}$, for $q= p \equiv \pm 3 \mod{8}$ and $q\not\equiv\pm 1\mod{10}$;
\item $\Alt_{5}$, for $q\equiv \pm1 \mod{10}$, where either $q=p$, or $q=p^2$ and $p \equiv \pm3 \mod{10}$;
\item $\D_{2(q-1)/\gcd(2, q -1)}$. If $H_{0}\cong \D_{q-1}$, then $q\geq 13$;
\item $\D_{2(q+1)/\gcd(2, q -1)}$. If $H_{0}\cong \D_{q+1}$, then $q\neq 7 ,9$;
\item $\C_{p}^{a}\rtimes \C_{q-1/\gcd(2, q-1)}$.
\end{enumerate}

\begin{table}
  \small
  \caption{$G$ and $H$ as in Proposition~\ref{prop:psu2}}\label{tbl:max-psu}
  \begin{tabular}{lllcl}
    \hline
    Line & $G$& $H$ &$|G:H|$ & $u_r$\\
    \hline
    $1$ & $\PGL(2,7)$ & $\N_{G}(\D_{6})=\D_{12}$ & $28$ & $3$\\
    $2$ & $\PGL(2,7)$& $\N_{G}(\D_{8})= \D_{16}$ & $21$ & $4$\\
    $3$ & $\PGL(2,9)$ & $\N_{G}(\D_{10})= \D_{20}$ & $36$ & $5$\\
    $4$ & $\PGL(2,9)$ & $\N_{G}(\D_{8})= \D_{16}$ & $45$ & $4$\\
    $5$ & $\M_{10}$ & $\N_{G}(\D_{10})= \C_{5}\rtimes  \C_{4}$ & $36$ & $5$\\
    $6$ & $\M_{10}$ & $\N_{G}(\D_{8})= \C_{8} \rtimes  \C_{2}$ & $45$ & $4$\\
    $7$ & $\PGaL(2,9)$ & $\N_{G}(\D_{10})= \C_{10} \rtimes  \C_{4}$ & $36$ & $5$\\
    $8$ & $\PGaL(2,9)$ & $\N_{G}(\D_{8})= \C_{8} \cdot \Aut(\C_{8})$ & $45$ & $4$\\
    $9$ & $\PGL(2,11)$ & $\N_{G}(\D_{10})= \D_{20}$ & $66$ & $5$\\
    $10$ & $\PGL(2,q)$, $q=p \equiv \pm 11,19 \mod{40}$ & $\N_{G}(\A_{4})= \Sym_{4}$ & $\frac{q(q^{2}-1)}{24}$ & -\\
     \hline
  \end{tabular}
\end{table}

We first show that $H_{0}$ must be maximal in $X$. If $H_{0}$ is not maximal in $X$, then one of the rows of Table~\ref{tbl:max-psu} holds. For each such $(G, H)$, by \eqref{eq:v}, we obtain $v$ as in the fourth column of Table \ref{tbl:max-psu}. Let $v\neq q(q^2-1)/24$. By Lemmas \ref{lem:six}(a) and (c), we can obtain an upper bound $u_r$ of $r$ as in the fifth column of Table \ref{tbl:max-psu}, and then we easily observe that $v>u_{r}^2$, and this violates Lemma \ref{lem:six}(c). \smallskip

Let now $v =q(q^2-1)/24$. Then by Lemmas~\ref{lem:New} and~\ref{lem:six}(c),  $r$ divides $|H|=24$. It follows from Lemma~\ref{lem:six}(c) that $v<r^2$, $q(q^2-1)\leq 24v<24r^2\leq {24}^{3}=6912$, and so $q\leq 24$. Thus $q=p=11$ or $19$ in which cases by \eqref{eq:v}, we have that  $v=55$ or $285$, respectively. By calculation, as $r$ divides $(|H|, v-1)$, we must have $r\leq 6$ in both cases, and so $v>r^2$, which is a contradiction. Therefore, $H_{0}$ is maximal in $X$.  We will analyze each of these cases separately.\smallskip

\noindent \textbf{(1)} In this case \eqref{eq:v} implies that $v=q_{0}(q_{0}^2+1)/2$. It follows from Lemmas \ref{lem:six}(d) and \ref{lem:subdegree} that $r$ divides $q_{0}(q_{0}^2-1)$. On the other hand, Lemma~\ref{lem:six}(a) implies that, $r$ divides $v-1$. As $v-1$ is coprime to $q_{0}$ by Lemma~\ref{lem:Tits} and $\gcd(q_{0}+1, q_{0}^{2}+q_{0}+2)=2$,  $r$ divides $2(q_{0}-1)$. Then by Lemma \ref{lem:six}(c) we have that $q_{0}(q_{0}^2+1)<4(q_{0}-1)^2$, which is impossible.\smallskip

\noindent \textbf{(2)} Here $v=q_{0}^{t-1}(q_{0}^{2t}-1)/(q_{0}^{2}-1)$ by~\eqref{eq:v}. By Lemmas~\ref{lem:New} and~\ref{lem:six}(c), $r$ must divide $a\cdot q_{0}({q_{0}}^2-1)$. Note that $v<r^{2}$ by Lemma~\ref{lem:six}(c). Moreover, $a^2<q=q_{0}^{r}$ and $q_{0}^2(q_{0}^2-1)^3< q_{0}^{8}$. Thus $q_{0}^{t-1}(q_{0}^{2t}-1)<(q_{0}^2-1)r^2 \leq a^2q_{0}^2(q_{0}^2-1)^3 <q_{0}^{8+t}$, and so $q_{0}^{2t-1}<q_{0}^{9}$. Hence $t<5$. As $t$ is odd, we conclude that $t=3$. Then $v={q_{0}}^2({q_{0}}^4+{q_{0}}^2+1)$.  By Lemma~\ref{lem:six}(c), $r$ divides $\gcd(v-1, aq_{0}({q_{0}}^2-1))$. It follows from Lemma~\ref{lem:Tits} that $v-1$ and $q_{0}$ are coprime, and so $r$ divides $\gcd(v-1, a({q_{0}}^2-1))$. Since $v={q_{0}}^2({q_{0}}^4+{q_{0}}^2+1)$, we have that  $\gcd(v-1, {q_{0}}^2-1) =2$, and so $r$ must divide $2a$. Therefore, Lemma~\ref{lem:six}(c) implies that $v<r^{2}\leq 4a^{2}$, and hence ${q_{0}}^2({q_{0}}^4+{q_{0}}^2+1)<4a^{2}\leq 36 s^2$, where $a=3s$ and $q_{0}=p^s$. Hence, $q_{0}^6<36s^2$, which is impossible.\smallskip

\noindent \textbf{(3)} In this case by \eqref{eq:v}, we have that $v=q_{0}^{t-1}(q_{0}^{2t}-1)/(q_{0}^{2}-1)$. Our argument here is the same as in the proof of \textbf{($2$)}. By Lemmas~\ref{lem:New} and~\ref{lem:six}(c), we have that
\begin{align}\label{eq:psu2-3-r}
  r \text{ divides } a{\cdot}q_{0}(q_{0}^2-1).
\end{align}
Since $a^2<2q=2q_{0}^t$ and $q_{0}^2({q_{0}}^2-1)^2<q_{0}^6$, it follows from Lemma~\ref{lem:six}(c) that $2q_{0}^{t-1}(q_{0}^{2t}-1)\leq (q_{0}^2-1)\lambda v< r^2 \leq a^2q_{0}^2(q_{0}^2-1)^2<2q_{0}^{6+t}$, and so $q_{0}^{2t-1}<q_{0}^{9}$. Therefore, $t<5$, and hence $t=2, 3$.

If $t=2$, then $v=q_{0}(q_{0}^2+1)$, where $q=q_{0}^2=2^{2s}$ and $2s=a$. It follows from ~\eqref{eq:psu2-3-r} and Lemma~\ref{lem:six}(a), $r$ divides $\gcd(v-1, a\cdot q_{0}({q_{0}}^2-1))$. Since $v-1$ and $q_{0}$ are coprime by Lemma~\ref{lem:Tits}, $r$ must divide $\gcd(v-1, a\cdot({q_{0}}^2-1))$. Since also $v=q_{0}(q_{0}^2+1)$, it follows that $\gcd(v-1, q_{0}^2-1)=1$ or $3$, and so $r$ must divide $3a$. Therefore, Lemma~\ref{lem:six}(c) implies that $\lambda v<r^{2}\leq 9a^{2}$, and hence $2q_{0}(q_{0}^2+1)\leq \lambda v<r^2<9a^{2}\leq 36 s^2$, where $2s=a$ and $q_{0}=2^s\geq 4$. Hence, $2^{3s}<18s^2$. This inequality holds only for $s =2$, in which case $v=68$ and $r$ divides $\gcd(67, 60)=1$, which is a contradiction.

If $t=3$, then $v=q_{0}^2(q_{0}^4+q_{0}^2+1)$, where $q=q_{0}^3=2^{3s}$ and $3s=a$. Again by~\eqref{eq:psu2-3-r} and Lemma~\ref{lem:six}(a), $r$ must divide $\gcd(v-1, a\cdot q_{0}(q_{0}^2-1))$. Since $\gcd(q_{0}^2(q_{0}^4+{q_{0}}^2+1)-1, q_{0}(q_{0}^2-1))=1$, $r$ divides $a$. It follows from Lemma~\ref{lem:six}(c) that $v<r^{2}\leq a^{2}$, and hence $q_{0}^2(q_{0}^4+q_{0}^2+1)<a^{2}\leq 9s^2$, where $3s=a$ and $q_{0}=2^s\geq 4$. Hence, $2^{6s}<9s^2$, which is impossible.\smallskip

\noindent \textbf{(4)} Here by \eqref{eq:v}, we conclude that $v=q(q^2-1)/48$. By Lemmas~\ref{lem:New} and~\ref{lem:six}(c), $r$ divides $48$. Again by Lemma~\ref{lem:six}(c), we must have
$q^3-q < 48r^2 <(48)^3=110592$ implying that $q\leq 47$, and since $q \equiv \pm 1 \mod 8$, we conclude that $q=7, 17, 23, 31, 41, 47$. Thus $v=7, 102, 253, 620, 1435, 2162$, respectively. Since $r$ is a divisor of $48$, $k\leq r$ and $\lambda=r(k-1)/(v-1)$ is a positive integer by Lemma~\ref{lem:six}(a), we conclude that $(v, r, k,\lambda)$ is $(7, 6, 3, 2)$, $(7, 6, 4, 3)$, $(7, 6, 5, 4)$, $(253, 24, 22, 2)$, $(253, 48, 22, 4)$ or $(253, 48, 43, 8)$. These remaining cases can be ruled out as $\gcd(r, \lambda)=1$.\smallskip

\noindent \textbf{(5)} In this case by~\eqref{eq:v}, we have that $v=q(q^2-1)/24$. Our argument here is the same as in the proof of \textbf{($4$)}. It follows from Lemmas~\ref{lem:New} and~\ref{lem:six}(c) that $r$ must divide $24$. By Lemma~\ref{lem:six}(c), $q(q^2-1)\leq 24 v< 24r^2 \leq 24^3$. This inequality holds only for $q=3$, $5$ or $13$, for which case $v=1$, $5$ or $91$, respectively. This case can be rule out as for each divisor $r$ of $24$ and each $v$ as above and $k\leq r$, the value of $\lambda=r(k-1)/(v-1)$ is not a positive integer, except the case where $(v, r, k, \lambda)=(5, 4, 3, 2)$ or $(91, 24, 16, 4)$ in which cases $\gcd(r, \lambda) \neq 1$, which is a contradiction.\smallskip

\noindent \textbf{(6)} By \eqref{eq:v}, we have that $v= q(q^2-1)/120$. Note by Lemmas~\ref{lem:New} and~\ref{lem:six}(c) that $r$ divides $120a$ (for $a=1, 2$). It follows from Lemma~\ref{lem:six}(c) that
\begin{align}\label{ss:2}
  \frac{q(q^2-1)}{120}\leq \lambda v<r^{2}\leq 120^{2}a^{2},
\end{align}
for $a=1, 2$.

If $a=1$, then $q=p< 120$ by~\eqref{ss:2}. Since $v=q(q^2-1)/120$ and $q\equiv \pm1 \mod{10}$, we must have $(v,q)\in \{(11, 11)$, $(57, 19)$, $(203, 29)$, $(248, 31)$, $( 574, 41)$, $(1711, 59)$, $(1891, 61)$, $(2982, 71)$, $(4108, 79)$, $(5874, 89)$, $(8585,101)$, $(10791, 109)\}$. Then for each divisor $r$ of $120$ and $v$ as above, we connot obtain any possible parameters, which is a contradiction.

If $a=2$, then by \eqref{ss:2}, $q=p^2<190$, and by the same argument as above, we conclude that  $v$ is $6$, $980$ or $40222$, and so, for each such $v$,  each divisor $r$ of $240$ and each parameter $k<v-1$, by Lemma~\ref{lem:six}, we easily obtain the parameters $(v,b,r,k,\lambda)=(6, 10, 5, 3, 2)$. By \cite{a:Zhou-nonsym-alt}, there is no flag-transitive $2$-$(6, 3, 2)$ design with $X=\PSU(2, 9)\cong \A_{6}$.\smallskip

\noindent \textbf{(7)} In this case by~\eqref{eq:v}, we conclude that $v=q(q+1)/2$. It follows from Lemmas~\ref{lem:six}(c) and~\ref{lem:subdegree} that $r$ divides $(q-1)/\gcd(2, q-1)$. Then Lemma \ref{lem:six}(c) implies that $\lambda\leq 2(q-1)^2/\gcd(2, q-1)^2q(q+1)$. Therefore, $\lambda=1$, which is a contradiction.\smallskip

\noindent \textbf{(8)} Here by~\eqref{eq:v}, we have that $v=q(q-1)/2$ . It follow from Lemmas~\ref{lem:six}(c) and~\ref{lem:subdegree}  that $r$ divides $q+1$. Let $m$ be a positive integer such that $mr=q+1$. Since $k\leq r$ and $\lambda>1$, by Lemma~\ref{lem:six}(a), we have that $2(q+1)-2m>m^2(q-2)$, and so $m^2(q-2)+2m<2(q+1)$. This inequality holds only for $m=1$. Recall that $mr=q+1$.  If $m=1$, then $r=q+1$ and $\lambda= 2$ by Lemma~\ref{lem:six}(a) as $q\neq 4,5$. Let $\lambda=2$. Then $q$ is even as $\gcd(r, \lambda)=(q+1, 2)=1$.  Also $k=q-1$ and $b=q(q+1)/2$ by Lemma \ref{lem:six}(a) and (b). Thus, the design $\Dmc$ has parameters $(v, b, r, k, \lambda)=$ $(q(q-1)/2, q(q+1)/2, q + 1, q-1, 2)$. As $r=k+\lambda$ and $\lambda=2$, $\Dmc$ can be embedded in a symmetric $((q^2+q+2)/2, q + 1, 2)$ design $\Dmc^{'}=(\Pmc^{'}, \Bmc^{'})$ by Lemma \ref{lem:symmetric}. We here use the same notation as in \cite{b:dembowski}  and use the same argument as in \cite[Lemma 3.3]{a:Non-symmetric}.  Let $B_0$ be a block of $\Dmc'$ of size $q+1$, and let $\Dmc$ be a design whose point set is $\Pmc'\setminus B_{0}$ and blocks set contains all blocks in $B_{0}\neq B'\in \Bmc'$. Note that  $G$ acts transitively on $\Bmc^{'}\setminus B_0$ and $G$ acts $2$-homogenously on the points of $B_0$. Hence, $G$ can be regarded as an automorphism group on $\Dmc$, and the action of $G$ on points of $B_0$ is isomorphic to the action of $G$ on the Sylow $2$-subgroups by conjugation, see \cite[Lemma 3.3]{a:Non-symmetric}.

Let $B$ be a block of $\Dmc$ incident with point $\alpha$. Let also $M$ be a maximal subgroup of $X$ containing $X_{B}$. Since $|X: X_{B}|$ divides $b$, $|X:M|$ must divide $q(q+1)/2$. A direct inspection of the list of the maximal subgroups $M$ of $X$ that $|X:M|$ divide $q(q+1)/2$ shows that $M$ is isomorphic either to $\C_{p}^{a}\rtimes \C_{q-1}$ or $\D_{2(q-1)}$. If $M=\D_{2(q-1)}$, then by the same method as in \cite[Lemma 3.3]{a:Non-symmetric}, we have that $(v, b, r, k, \lambda)=(28, 36, 9, 7, 2)$ and $X=\PSU(2, 8)$. By \cite[Table B.2]{b:Dixon}, $X$ has a faithful $2$-transitive action of degree $28$, so $X$ is flag-transitive by Lemma \ref{lem:flag}. Then by \cite{a:Non-symmetric}, $\Dmc$ is a unique $2$-$(28, 7, 2)$ design admitting $G=\PSU(2, 8)$ as its flag-transitive automorphism group. If $M=\C_{p}^{a}\rtimes \C_{q-1}$, then $2(q-1)$ divides $|X_{B}|$ as $|X: X_{B}|$ divides $b$. By inspection of the list of the subgroups of $X$ whose order is divisible  $2(q-1)$, we conclude that $X_{B}$ is isomorphic to $\D_{2(q-1)}$ or $\C_{p}^{a}\rtimes \C_{q-1}$. Recall that $M=\C_{p}^{a}\rtimes \C_{q-1}$ is maximal in $X$ and $X_{B}\leq M$. Therefore, $X_{B}=\C_{p}^{a}\rtimes \C_{q-1}$. Note that $|X_{\alpha}:X_{\alpha, B}|$ divides $q+1$ and $|X_{\alpha}|=2(q+1)$. Then $|X_{\alpha, B}|=2s$, where $s$ divides $q+1$. As $X_{\alpha, B}\leq X_{B}=\C_{p}^{a}\rtimes \C_{q-1}$, we conclude that $s=1$, and so $|X_{\alpha, B}|=2$,  and this contradicts the fact that $|X_{B}:X_{\alpha, B}|$ divides $q-1$.\smallskip

\noindent \textbf{(9)} Here by \eqref{eq:v}, we have that $v=q+1$. By \cite[p.245]{b:Dixon}, $X$ is $2$-transitive, so $X$ is flag-transitive by Lemma \ref{lem:flag}. Hence we may assume that $G=X=\PSU(2, q)$.  Then by \cite{a:Non-symmetric}, up to isomorphism $\Dmc$ is the unique $2$-$(8, 4, 3)$ design with $G=\PSU(2, 7)$. This completes the proof.
\end{proof}

\begin{proposition}\label{prop:psu}
Let $\Dmc$ be a nontrivial $2$-design with $\gcd(r, \lambda)=1$. Suppose that $G$ is an automorphism group of $\Dmc$ whose socle is $X=\PSU(n, q)$ with $n\geq 3$ and $(n, q)\neq (3, 2)$. If $G$ is flag-transitive, then $\Dmc$ is a Hermitian unital with parameters $(q^{3}+1,q+1,1)$ and $X=\PSU(3, q)$.
\end{proposition}
\begin{proof}
Suppose that $H_{0}=H\cap X$, where $H=G_{\alpha}$ with $\alpha$ a point of $\Dmc$. Since the point-stabiliser $H$ is maximal in $G$, by Proposition~\ref{prop:large} one of the following holds:
\begin{enumerate}[\rm (1)]
\item $H \in \Cmc_{1}$;
\item $H$ is a $\Cmc_2$-subgroup of type $\GU(m, q)\wr\S_t$ with $m=1$ or $m\geq2$ and $t\leq 11$;
\item $H$ is a $\Cmc_2$-subgroup of type $\GL(n/2, q^2)$;
\item $H$ is a $\Cmc_3$-subgroup of type $\GU(m, q^t)$ with $mt=n$;
\item $H$ is a $\Cmc_5$-subgroup of type $\GU(n, q_{0})$ with $q=q_{0}^3$;
\item $H$ is a $\Cmc_5$-subgroup of type $\Sp(n, q)$ or $\O^{\e}(n, q)$ with $\e {\in} \{\circ, -, +\}$.
\end{enumerate}
In what follows, we discuss each of these possible cases. \smallskip

\noindent\textbf{(1)} In this case $H$ is reducible and it is either a parabolic subgroup $P_{m}$, or the stabiliser $N_{m}$ of a nonsingular subspace. Suppose first that $H_{0}$ is isomorphic to $P_{m}$, for some $2m\leq n$. Then by Lemma~\ref{lem:subdeg}, there is a unique subdegree which is a power of $p$. Note that the highest power of $p$ dividing $v-1$ is at most $q^3$. If $n\neq 3$, then by \eqref{eq:v}, we have that $v>q^{m(2n-3m)}$ and so $v>r^2$, which is a contradiction. If $n=3$, then the action is $2$-transitive, and this case has already been done by Kantor~\cite{a:kantor-Homdes} and $\Dmc$ is a Hermitian unital  with parameters $(q^{3}+1,q+1,1)$ and $X$ is $\U_3(q)$.

Suppose now that $H_{0}$ is isomorphic to $N_{m}$ with $2m <n$. Then by~\cite[Proposition 4.1.4]{b:KL-90}, $H_{0}$ is isomorphic to $\,^{\hat{}}\SU(m, q) \times \SU(n-m, q) \cdot (q+1)$. Here by~\eqref{eq:v}, we have that $v>q^{m(n-m)}$. By Lemma~\ref{lem:subdegree}, we see that $r$ divides $(q^{m}-(-1)^m)(q^{n-m}-(-1)^{n-m})$, and since $r^{2}>v$, we have that $m=1$. Then $r$ divides $(q+1)(q^{n-1}-(-1)^{n-1})$. By Lemma \ref{lem:divisible}, we see that $r$ is divisible by the degree of a parabolic action of $\U_{n-1}(q)$. If $n\geq 5$ and $(n, q)\neq (6s, 2)$, then by Lemma \ref{lem:min-deg}, we have that $(q^5-1)\leq(q^n-1)<(q+1)(q^2-1)$, which is impossible. If $n\equiv 0 \mod 6$ and $q=2$, then by Lemma \ref{lem:min-deg}, $2^{n-2}<2^{n-1}(2^{n}-1)\leq 9(2^{n-1}+1)$, and so $2^{2n-2}<2^{n+3}$. Thus $n<5$, which is impossible. If $n=4$, then $v=q^3(q^2+1)(q-1)$. In this case by \cite[Corollary 3]{a:Korab-LnUn}, $q^3+1$ must divide $r$. Note that $\gcd(v-1, q^3+1)$ divides $q^2-q+1$. Then $q^3+1$ must divide $q^2-q+1$, which is impossible. If $n=3$, we again use~\cite[Corollary 3]{a:Korab-LnUn} and conclude that $q+1$ divides $r$. But $\gcd(q^4-q^3+q^2-1, q+1)$ is at most $2$, and so $q+1$ divides $2$, which is impossible.\smallskip

\begin{table}
\small
\caption{Some large maximal $\Cmc_{i}$-subgroups with $i=2, 3$ of some unitary groups.}\label{tbl:2-n-q-v-r}
\begin{tabular}{cllll}
\hline\noalign{\smallskip}
\multicolumn{1}{c}{Class} &
\multicolumn{1}{l}{$X$} &
\multicolumn{1}{l}{$H\cap X$} &
\multicolumn{1}{l}{$v$} &
\multicolumn{1}{c}{$u_{r}$}
\\
\noalign{\smallskip}\hline\noalign{\smallskip}
$\Cmc_{2}$ &
$\PSU(4, 2)$ &
$3^3\cdot \Sym_{4}$ &
$40$&
$3$\\
$\Cmc_{2}$ &
$\PSU(4, 3)$ &
$4^3\cdot \Sym_{4}$ &
$8505$&
$8$\\
$\Cmc_{2}$ &
$\PSU(4,4)$ &
$5^3\cdot \Sym_{4}$ &
$339456$&
$5$\\
$\Cmc_{2}$ &
$\PSU(4,5)$ &
$6^3\cdot \Sym_{4}$ &
$5687500$&
$3$\\
$\Cmc_{2}$ &
$\PSU(4,7)$ &
$8^3\cdot \Sym_{4}$ &
$379418025$&
$8$\\
$\Cmc_{2}$ &
$\PSU(4,8)$ &
$9^3\cdot \Sym_{4}$ &
$1982955520$&
$27$\\
$\Cmc_{2}$ &
$\PSU(4,9)$ &
$10^3\cdot \Sym_{4}$ &
$8483215536$&
$5$\\
$\Cmc_{2}$ &
$\PSU(4,11)$ &
$12^3\cdot \Sym_{4}$ &
$99960329425$&
$48$\\
$\Cmc_{2}$ &
$\PSU(4,13)$ &
$14^3\cdot \Sym_{4}$ &
$772965193260$&
$7$ \\
$\Cmc_{2}$ &
$\PSU(4,16)$ &
$17^3\cdot \Sym_{4}$ &
$9741847756800$&
$17$\\
$\Cmc_{2}$ &
$\PSU(4,17)$ &
$18^3\cdot \Sym_{4}$ &
$20383694269120$&
$27$ \\
$\Cmc_{2}$ &
$\PSU(4,19)$ &
$20^3\cdot \Sym_{4}$ &
$78860234613321$&
$40$\\
$\Cmc_{2}$ &
$\PSU(4,23)$ &
$24^3\cdot \Sym_{4}$ &
$802204261952665$&
$24$\\
$\Cmc_{2}$ &
$\PSU(5,2)$ &
$3^4\cdot \Sym_{5}$ &
$1408$&
$3$\\
$\Cmc_{2}$ &
$\PSU(5,3)$ &
$4^4\cdot \Sym_{5}$ &
$8404641$&
$160$\\
$\Cmc_{2}$ &
$\PSU(6,2)$ &
$3^5\cdot \Sym_{6}$ &
$157696$&
$15$\\
$\Cmc_{3}$ &
$\PSU(3, 2)$ &
$3{\cdot}3$ &
$8$ &
$1$
\\
$\Cmc_{3}$ &
$\PSU(3, 3)$ &
$7{\cdot}3$ &
$288$ &
$7$ \\
$\Cmc_{3}$ &
$\PSU(3, 4)$ &
$13{\cdot}3$ &
$1600$ &
$39$ \\
$\Cmc_{3}$ &
$\PSU(3, 8)$ &
$57{\cdot}3$ &
$32256$ &
$1$ \\
$\Cmc_{3}$ &
$\PSU(3, 16)$ &
$241{\cdot}3$ &
$5918720$ &
$241$ \\
\hline
\end{tabular}
\end{table}

\noindent\textbf{(2)} Let $H$ be a $\Cmc_2$-subgroup of type $\GU(m, q)\wr\S_t$, where $t\leq 11$. In this case $H$ preserves a partition $V=V_{1}\oplus \ldots \oplus V_{t}$ with each $V_{i}$ of dimension $m$, so $n=mt$, and the $V_{i}$'s are non-singular and the partition is orthogonal. Then by~\cite[Proposition 4.2.9]{b:KL-90}, $H_{0}$ is isomorphic to $\,^{\hat{}}\SU(m, q)^t{\cdot}(q+1)^{t-1}{\cdot}\Sym_{t}$. It follows from Lemma~\ref{lem:up-lo-b} and~\eqref{eq:v} that $v>q^{(m^2t-2)(t-1)}/(t!)$. If $m>1$, then by Lemma~\ref{lem:subdegree}, we see that $r$ divides $t(t-1)(q^m-(-1)^m)^2$. Then the inequality $v<r^2$ implies that $(m, t)=(2, 2)$. Then by~\eqref{eq:v}, we have that $v=q^4(q^2+1)(q^2-q+1)/2$ and $r$ divides $2(q^2-1)^2$. Since $\gcd(v-1, q+1)=(2, q+1)$, $r$ divides $8(q-1)^2$. It follows from Lemma \ref{lem:six}(c) that $q^4(q^2+1)(q^2-q+1)<128(q-1)^4$. This inequality does not hold for any $q=p^a$, which is a contradiction. Therefore, $m=1$ and $H_{0}$ is isomorphic to $\,^{\hat{}}(q+1)^{n-1}\cdot \Sym_{n}$. If $n=3$ with $q>2$, then by~\eqref{eq:v}, we have that $v=q^3(q^2-q+1)(q-1)/6$. By Lemmas~\ref{lem:New} and \ref{lem:six}(c), $r$ divides $12a(q+1)^2$. Since $r^2 > v$, we have that $q^3(q^2-q+1)(q-1)<864a^2(q+1)^4$. This inequality holds when:
\begin{align*}
  \begin{array}{llll}
    p =2, & \quad a\leq 7; \\
    p =3, & \quad a\leq 4; \\
    p =5, 7 & \quad a\leq 2; \\
    p =11, 13, 17, 19, 23, 29, 31& \quad a= 1.
 \end{array}
\end{align*}
These remaining cases are easily excluded, using the facts that $r$ divides $\gcd(v-1, 12a(q+1)^2)$ and $v<r^2$. Hence $n\geq 4$ and~\eqref{eq:v} implies that
\begin{align}\label{eq:2-v}
v=\frac{q^{n(n-1)/2}(q^{n}-(-1)^{n})\cdots(q^2-1)}{(q+1)^{n-1} \cdot n!}.
\end{align}
It follows from Lemma~\ref{lem:subdegree} that $r$ divides $d=n(n-1)(n-2)(q+1)^3/2$. Then Lemma~\ref{lem:six}(c) implies that $q^{n(n-1)/2}(q^{n}-(-1)^{n})\cdots(q^2-1)<(n!)\cdot n^2(n-1)^2(n-2)^2(q+1)^{n+5}$. Note that $q^{n(n-1)/2}<(q^{n}-(-1)^{n})\cdots(q^2-1)$. Therefore, $q^{n^2-2n-5}<(n!)\cdot n^2(n-1)^2(n-2)^2$. This inequality holds only for pairs $(n, q)\in\{(4, 2)$, $(4, 3)$, $(4, 4)$, $(4, 5)$, $(4, 7)$, $(4, 8)$,  $(4, 9)$, $(4, 11)$, $(4, 13)$, $(4, 16)$, $(4, 17)$, $(4, 19)$, $(4, 23), $(5, 2)$, $(5, 3)$, $(6, 2)$\}$. For each such $(n, q)$, by~\eqref{eq:2-v} and Lemma~\ref{lem:six}(c), we obtain $v$ an upper bound $u_r=\gcd(v-1, d)$ of $r$ as in the fourth and fifth column of Table~\ref{tbl:2-n-q-v-r}, and then we easily observe that $v>u_{r}^2$, and this violates Lemma~\ref{lem:six}(c).\smallskip

\noindent\textbf{(3)} Let $H$ be a $\Cmc_2$-subgroup of type $\GL(n/2, q^2)$. Then by~\cite[Proposition 4.2.4]{b:KL-90}, $H_{0}$ is isomorphic to $\,^{\hat{}}\SL(m, q^2)\cdot(q-1)\cdot 2$, where $2m=n$. From~\eqref{eq:v}, we conclude that $v=q^{m^2}(q^{2m-1}+1)(q^{2m-3}+1){\cdots}(q^3+1)(q+1)/2$. It follows from Lemma~\ref{lem:subdegree} that $r$ must divide $2(q^{2m}-1)$. Since $r^2 > v$, we have that $m=2$. Then $v=q^4(q^3+1)(q+1)/2$. By Lemma~\ref{lem:six}(a), $r$ divides $v-1$, and so $r$ must divide $\gcd(v-1, 2(q^4-1))$. Note that $\gcd(v-1, 2(q+1))=1$. Thus $r$ divides in fact $(q^4-1)/(q+1)$, so again $r^2<v$, which is a contradiction.\smallskip

\noindent\textbf{(4)} Let $H$ be  a $\Cmc_3$-subgroup of type $\GU(m, q^t)$ with $mt=n$. Then by~\cite[Proposition 4.3.6]{b:KL-90}, $H_{0}$ is isomorphic to $\,^{\hat{}}\SU(m, q^t)\cdot(q^t+1)\cdot t/(q+1)$. We now apply Lemma~\ref{lem:up-lo-b}. Then the inequality $|X|<|\Out(X)|^2 \cdot |H_{0}| \cdot |H_{0}|_{p'}^2$ implies that
\begin{align}\label{eq:3-m-t}
q^{m^2t^2}<32a^2t^3\cdot q^{m^2t+3t-1}\cdot \prod_{i=2}^{m}(q^{it}-(-1)^i)^2.
\end{align}
Note that $a^2<2q$. Then Lemma \ref{lem:equation} implies that $q^{m^2t(t-2)-t(m+1)}<64{\cdot}t^3$. This inequality holds only for $(m, t)\in\{(1, 3)$, $(1, 5)$, $(2, 3)\}$. If $(m, t)=(2, 3)$, then~\eqref{eq:3-m-t} implies that $q^{36}<864a^2 q^{20}(q^6-1)^2$. Therefore $q^{16}<864a^2(q^6-1)^2$. This inequality holds only for $q=2, 3, 4, 5$ or $8$. For each of these $q$, we can obtain $v$ by~\eqref{eq:v}, and by considering the fact that $v< r^2\leq \gcd(|H|, v-1)^2$, we cannot obtain any possible design. If $(m, t)=(1, 5)$, then~\eqref{eq:3-m-t} yields $q^{6}<4000a^2$. This inequality holds only when $q=2, 3$ or $4$, in which cases by \eqref{eq:v}, $v=248832, 846526464$ or $260702208000$, respectively. Again for which by considering the fact that $v<r^2\leq \gcd(|H|, v-1)^2$, we cannot obtain any possible design. If $(m, t)=(1, 3)$, then by \eqref{eq:v}, we have that $v=q^3(q^2-1)(q+1)/3$. Note that $r$ divides $v-1$. Hence $r$ is odd. On the other hand Lemmas \ref{lem:New} and \ref{lem:six}(c) implies that $r$ divides $6a(q^2-q+1)$. Thus $r\leq 3a(q^2-q+1)$, and as $r^2 > v$, we have that $q^3(q^2-1)(q+1)<27a^2(q^2-q+1)^2$. This inequality holds only for $q=2, 3, 4, 8, 16$. For each such value of $q$, by~\eqref{eq:v}, we obtain $v$ as in the third column of Table \ref{tbl:2-n-q-v-r}. Moreover, Lemma \ref{lem:six}(a) and (c) say that $r$ divides $\gcd(|H|, v-1)$ , and so we can find an upper bound $u_r$ of $r$ as in the fourth column of Table \ref{tbl:2-n-q-v-r}. Then the inequality $v<r^2$ rules out all cases.\smallskip

\noindent\textbf{(5)} Let $H$ be  a $\Cmc_5$-subgroup of type $\GU_{n}(q_{0})$ with $q=q_{0}^t$ and $t$ an odd prime. Then by~\cite[Proposition 4.5.3]{b:KL-90}, $H_{0}$ is isomorphic to $\,^{\hat{}}\SU(n, q_{0})\cdot \gcd((q+1)/(q_{0}+1), n)$. Here By Lemma~\ref{lem:up-lo-b} and the inequality $|X|<|\Out(X)|^2{\cdot}|H_{0}|{\cdot}|H_{0}|_{p'}^2$, we have that
\begin{align}\label{eq:5-q-n-t}
q_{0}^{t(n^2-2)}<4a^2n^3\cdot q_{0}^{n^2+2}\prod_{i=2}^{n}(q_{0}^i-(-1)^i)^2.
\end{align}
Then by Lemma \ref{lem:equation} and the fact that $a^2<2q$, we have that $q_{0}^{n^2(t-2)-n-3t}<8n^3$, so either $(n, t)=(3, 3)$ or $(n, t)=(4, 3)$. If $(n, t)=(4, 3)$, then by \eqref{eq:v}, we have that $v=q_{0}^{12}(q_{0}^{12}-1)(q_{0}^{9}+1)(q_{0}^{6}-1)/(q_{0}^{4}-1)(q_{0}^{3}+1)(q_{0}^{2}-1)$. By Lemmas \ref{lem:New} and \ref{lem:six}(c), $r$ divides $2aq_{0}^{6}(q_{0}^{4}-1)(q_{0}^{3}+1)(q_{0}^{2}-1)$. Since $r$ divides $v-1$ by Lemma \ref{lem:six}(a) and $v-1$ is coprime to $q_{0}$ by Lemma~\ref{lem:Tits}, $r$ must divide $2a (q_{0}^{4}-1)(q_{0}^{3}+1)(q_{0}^{2}-1)$. Now Lemma~\ref{lem:six}(c) implies that $q_{0}^{12}(q_{0}^{12}-1)(q_{0}^{9}+1)(q_{0}^{6}-1)<4a^2(q_{0}^{4}-1)^3(q_{0}^{3}+1)^3(q_{0}^{2}-1)^3$. Therefore, $q_{0}^{12}<4a^{2}$, which is impossible. If $(n, t)=(3, 3)$, then \eqref{eq:v} implies that $v=q_{0}^6(q_{0}^6-q_{0}^3+1)(q_{0}^4+q_{0}^2+1)$. It follows from Lemmas \ref{lem:New} and \ref{lem:six}(c) that $r$ divides $2aq_{0}^{3}(q_{0}^{3}+1)(q_{0}^{2}-1)$. Then Lemmas \ref{lem:six}(a) and~\ref{lem:Tits} implying that $r$ must divide $2a(q_{0}^{3}+1)(q_{0}^{2}-1)$. Now Lemma~\ref{lem:six}(c) implies that $q_{0}^{9}(q_{0}^{9}+1)(q_{0}^{6}-1)<4a^2(q_{0}^{3}+1)^3(q_{0}^{2}-1)^3$. Thus, $q_{0}^{9}<4a^{2}$, which is impossible.\smallskip

\noindent\textbf{(6)} Let $H$ be a $\Cmc_5$-subgroup of type $\Sp_n(q)$ or $\O_{n}^{\e}(q)$. If $H$ is of type $\O_{n}^{\e}(q)$ with $n$ even and $q$ odd, then by~\cite[Proposition 4.5.5]{b:KL-90}, $H_{0}$ is isomorphic to $\PSO_{n}^{\e}(q){\cdot}2$ and by \eqref{eq:v}, we conclude that $v=q^{m^2}(q^{m}+\e1)(q^{2m-1}+1)(q^{2m-3}+1){\cdots}(q^3+1)/\gcd(2m, q+1)$, where $2m=n$. By the Tits' Lemma~\ref{lem:Tits}, $r$ is divisible by the degree of some parabolic action of $H$. Here $q+1$ divides $r$ and hence $r$ is even. On the other hand $v-1$ is odd except for the case where $(n, \e)=(4, +)$, so that is impossible. Assume now that $(n, \e)=(4, +)$. Then $q+1$ divides $r$, and $(q+1)/\gcd(4, q+1)$ divides $v$. Then the fact that $r$ divides $v-1$ implies that $q=3$. In which case $v=2835$ and $r\leq \gcd(v-1, |H|)=2$. Therefore $v>r^2$, which is a contradiction. If $H$ is of type $\Sp(n, q)$ with $n$ even, then by~\cite[Proposition 4.5.6]{b:KL-90}, $H_{0}$ is isomorphic to $\,^{\hat{}}\Sp(n, q){\cdot}\gcd(n/2, q+1)$. Here by \eqref{eq:v}, we have that $v=q^{m^2-1}(q^{2m-1}+1)(q^{2m-3}+1){\cdots}(q^3+1)/\gcd(m, q+1)$, where $2m=n$. It follows from Tits' Lemma~\ref{lem:Tits}, $r$ is divisible by the degree of some parabolic action of $H$. In this case $q+1$ divides $r$. On the other hand, it is easy to see that $v$ is divisible by $(q+1)/\gcd(m, q+1)$. This contradicts the fact that $r$ divides $v-1$.\smallskip
\end{proof}

\begin{proof}[\rm \textbf{Proof of Theorem~\ref{thm:main}}]
The proof of the main result follows immediately from Propositions \ref{prop:psu2} and \ref{prop:psu}.
\end{proof}

\section*{Acknowledgements}

The authors would like to thank anonymous referees for providing us helpful and constructive comments and suggestions.



\end{document}